\documentclass[11pt]{amsart}
\usepackage{amsmath,amsthm,amssymb,tikz,mathrsfs,pgfplots}
\usepackage[numbers,comma,square,sort&compress]{natbib}
\usepackage{enumitem}

\RequirePackage{color}
\definecolor{my-link}{rgb}{0.5,0.0,0.0}
\definecolor{my-blue}{rgb}{0.0,0.0,0.6}
\definecolor{my-red}{rgb}{0.5,0.0,0.0}
\definecolor{my-green}{rgb}{0.0,0.5,0.0}
\definecolor{nicos-red}{rgb}{0.75,0.0,0.0}
\definecolor{really-light-gray}{gray}{0.8}
\definecolor{darkgreen}{rgb}{0.0,0.5,0.0}
\definecolor{darkblue}{rgb}{0.0,0.0,0.3}
\definecolor{nicosred}{rgb}{0.65,0.1,0.1}
\definecolor{light-gray}{gray}{0.7}
\RequirePackage[colorlinks,urlcolor=my-blue,linkcolor=my-red,citecolor=my-green]{hyperref}

\usepackage{xargs}
\usepackage[colorinlistoftodos,prependcaption,textsize=footnotesize]{todonotes}
\newcommandx{\addmath}[2][1=]{\todo[linecolor=red,backgroundcolor=red!25,bordercolor=red,#1]{#2}}
\newcommandx{\fixtext}[2][1=]{\todo[linecolor=blue,backgroundcolor=blue!25,bordercolor=blue,#1]{#2}}
\newcommandx{\note}[2][1=]{\todo[linecolor=yellow,backgroundcolor=yellow!25,bordercolor=yellow,#1]{#2}}

\usepackage[color]{changebar}
\cbcolor{red}
\usepackage{bm}

\usepackage[left=1in,top=1in,right=1in,bottom=1in]{geometry}

\usepackage{scalerel}    
\usepackage{stmaryrd}   
\usepackage{mathabx}   
\usepackage{hyphenat}
\usepackage{scalerel}    

\usepackage{sidenotes}

\usepackage[left=1in,top=1in,right=1in,bottom=1in]{geometry}

\newtheorem{theorem}{Theorem}[section]
\newtheorem{lemma}[theorem]{Lemma}
\newtheorem{proposition}[theorem]{Proposition}
\newtheorem{corollary}[theorem]{Corollary}

\theoremstyle{definition}

\theoremstyle{remark}
\newtheorem{remark}[theorem]{Remark}

\allowdisplaybreaks[3]

\numberwithin{figure}{section}
\numberwithin{equation}{section}

\definecolor{sussexg}{rgb}{0,0.7,0.7}
\definecolor{sussexp}{rgb}{0.4,0,0.4}
\definecolor{sussexb}{rgb}{0.4,0.4,0.7}
\definecolor{mygray}{rgb}{0.75,0.75,0.75}

\newcommand{\E}{\mathbb{E}}

\newcommand{\cM}{\mathcal{M}}

\renewcommand{\P}{\mathbb{P}}

\newcommand{\R}{\mathbb{R}}

\newcommand{\Z}{\mathbb{Z}}
\newcommand{\N}{\mathbb{N}}
\newcommand{\X}{\mathbb{X}}

\newcommand{\cF}{\mathcal{F}}

\newcommand{\e}{\varepsilon}

\newcommand{\fl}[1]{\lfloor{#1}\rfloor}

\DeclareMathOperator{\bbI}{\mathbb{I}}

\DeclareMathOperator{\bbW}{\mathbb{W}}

\DeclareMathOperator{\bbZ}{\mathbb{Z}}

\DeclareMathOperator{\bfb}{\mathbf{b}}

\DeclareMathOperator{\bfm}{\mathbf{m}}

\DeclareMathOperator{\bfy}{\mathbf{y}}

\DeclareMathOperator{\sC}{\mathcal{C}}

\DeclareMathOperator{\sI}{\mathcal{I}}

\DeclareMathOperator{\sM}{\mathcal{M}}

\newcommand{\overbar}[1]{\mkern 1.7mu\overline{\mkern-2.7mu#1\mkern-2.7mu}\mkern 1.2mu}
\newcommand{\overbarA}[1]{\mkern 1.5mu\overline{\mkern-3mu#1\mkern-1.5mu}\mkern 1.5mu}
\def\Xbar{\overbar X}
\def\Abar{\overbarA A}
\def\Bbar{\overbar B}

\def\w{\omega}

\providecommand{\abs}[1]{\vert#1\vert}

\def\fe{\Lambda}

\DeclareMathOperator{\ext}{ext}    

\makeatletter
\DeclareRobustCommand{\cev}[1]{%
  \mathpalette\do@cev{#1}%
}
\newcommand{\do@cev}[2]{%
  \fix@cev{#1}{+}%
  \reflectbox{$\m@th#1\vec{\reflectbox{$\fix@cev{#1}{-}\m@th#1#2\fix@cev{#1}{+}$}}$}%
  \fix@cev{#1}{-}%
}
\newcommand{\fix@cev}[2]{%
  \ifx#1\displaystyle
    \mkern#22mu
  \else
    \ifx#1\textstyle
      \mkern#22mu
    \else
      \ifx#1\scriptstyle
        \mkern#22mu
      \else
        \mkern#22mu
      \fi
    \fi
  \fi
}

\makeatother

\font \mymathbb = bbold10 at 11pt
\newcommand{\one}{\mbox{\mymathbb{1}}}    

\usepackage{filemod}

\begin{document}

\title[Stationary directed polymers]{Uniqueness and ergodicity of stationary\\ directed polymers on $\Z^2$}

\author[C.~Janjigian]{Christopher Janjigian}
\address{Christopher Janjigian\\ University of Utah\\  Mathematics Department\\ 155 S 1400 E\\   Salt Lake City, UT 84112\\ USA.}
\email{janjigia@math.utah.edu}
\urladdr{http://www.math.utah.edu/~janjigia}
\thanks{C.\ Janjigian was partially supported by a postdoctoral grant from the Fondation Sciences Math\'ematiques de Paris while working at Universit\'e Paris Diderot.}

\author[F.~Rassoul-Agha]{Firas Rassoul-Agha}
\address{Firas Rassoul-Agha\\ University of Utah\\  Mathematics Department\\ 155S 1400E\\   Salt Lake City, UT 84112\\ USA.}
\email{firas@math.utah.edu}
\urladdr{http://www.math.utah.edu/~firas}
\thanks{F.\ Rassoul-Agha was partially supported by National Science Foundation grants DMS-1407574 and DMS-1811090}


\date{April 11, 2019}    

\begin{abstract}  
We study the ergodic theory of stationary directed nearest-neighbor polymer models on $\Z^2$, with i.i.d.\ weights. 
Such models are equivalent to specifying a stationary distribution on the space of weights and correctors that satisfy certain consistency conditions. We show that for prescribed weight distribution and corrector mean, there is at most one stationary polymer distribution which is ergodic under the $e_1$ or $e_2$ shift. Further, if the weights have more than two moments and the corrector  mean vector is an extreme point of the superdifferential of the limiting free energy, then the corrector distribution is ergodic under each of the $e_1$ and $e_2$ shifts.
\end{abstract}
\maketitle


\section{Introduction}

{\sl Directed polymers with bulk disorder}  were introduced in the statistical physics literature by Huse and Henley \cite{Hus-Hen-85} in 1985 to model the domain wall separating the positive and negative spins in the ferromagnetic Ising model with random impurities. These models have been the subject of intense study over the past  three decades; see the recent surveys \cite{Com-17, Com-Shi-Yos-04, Hol-09, Gia-07}.  
Much of the reason for the interest in these models is due to the conjecture that under mild hypotheses, such models 
are members of the Kardar-Parisi-Zhang (KPZ) universality class, which is an extremely wide-ranging class of models believed to have the same statistical and structural properties.  
See \cite{Cor-12,Cor-16,Hal-Tak-15,Qua-Spo-15,Qua-12} for recent surveys.

Much like characterization of the Guassian universality in terms of variants of the central limit theorem (CLT), 
the KPZ class is characterized by its scaling exponents and limit distributions. 
In the context of the directed polymer models we study in this paper, the conjecture is that the appropriately centered and normalized finite volume free energy
converges to  a limit distribution which is independent of the random environment that the polymer lives in. 
The KPZ scaling theory \cite{Kar-Par-Zha-86,Kru-Mea-Hal-92} predicts that the fluctuations around the limiting free energy are of the order of the cube root of the size of the volume, in contrast to the square root in the classical CLT.
Moreover, the limiting distribution is not Gaussian.
The effect of the environment is felt only through the centering and normalizing constants, which play a role similar to that of the mean and standard deviation in the CLT.

The scaling theory also predicts the values of these two non-universal constants if one has a complete description of the spatially-ergodic and temporally-stationary measures for the polymer model. See \cite{Spo-14} for an example of this computation in the setting of the semi-discrete polymer model introduced by O'Connell and Yor in \cite{Oco-Yor-01}. In the mathematics literature, stationary polymer models have been studied primarily in the context of solvable models, 
which have product-form stationary distributions.
The first such solvable stationary polymer model was the aforementioned model due to O'Connell and Yor. The first discrete model was introduced by Sepp\"al\"ainen in \cite{Sep-12-corr}. See also the models introduced in \cite{Bar-Cor-17, Cor-Sep-She-15,Thi-16,Thi-Dou-15} and studied further in \cite{Bal-Ras-Sep-18-,Cha-Noa-18-ejp,Cha-Noa-18-alea}. Such models remain to date the only polymer models for which the KPZ universality conjectures have been verified.

In the present paper, we investigate the ergodic theory of stationary directed polymers on the lattice $\bbZ^2$ with general i.i.d.\ weights and nearest-neighbor steps. 
The solvable model in \cite{Sep-12-corr} is an example of the type of model we study.
Specializing our results to the case where the limiting free energy is everywhere differentiable, 
we show that the ergodic distributions form a one-parameter family, indexed by the derivative of the free energy.
This differentiability is satisfied in the model in \cite{Sep-12-corr} and is conjectured to hold generally.
 As a consequence, our results imply that the ergodic 
 measures constructed in \cite{Sep-12-corr} are the only such measures in that model, which verifies that the hypotheses of the scaling theory are satisfied. More generally, we give conditions under which one can conclude that a stationary distribution is ergodic as well as conditions under which an ergodic measure is unique.\smallskip


Apart from being fundamental objects for the study of the scaling theory, classifying stationary ergodic polymer measures is important for addressing several other questions. We mention two.\smallskip

Our results on the classification of stationary and ergodic polymer measures can be reformulated 
in terms of a characterization of the stationary and ergodic 
{\sl global  physical solutions} to a discrete version of the viscous stochastic Burgers equation
	\[\partial_t U=\frac12\partial_{xx}U+\frac12\partial_x U^2+\partial_x\dot W,\]
where $\dot W$ is space-time white noise.
This connection is the focus of our companion paper \cite{Jan-Ras-19-pams-}.\smallskip

In the language of statistical mechanics, stationary directed polymer measures are in correspondence with shift-covariant {\sl semi-infinite Gibbs measures} which are consistent with the quenched point-to-point polymer measures. For a discussion of this point of view, we refer the reader to \cite{Jan-Ras-18-arxiv}.\smallskip

We close this introduction by giving an outline of the rest of the paper.
We introduce polymer measures and stationary polymer measures in Section \ref{sec:setting} then state our main results in Section \ref{sec:results}. 
In Section \ref{sec:prelims} we prove some preliminary lemmas and motivate the setting of Section \ref{sec:aux},
where we prove some auxiliary results that are then used in the proofs of the main theorems in Section \ref{sec:proofs}.



\section{Setting}\label{sec:setting}


\subsection{Random polymer measures}
Let $\Omega_0=\R^{\Z^2}$ and equip it with the product topology and product Borel $\sigma$-algebra $\cF_0$.
A generic point in $\Omega_0$  will be denoted by $\w$. Let $\{\w_x(\w):x\in\Z^2\}$  be
the natural {\sl coordinate projections}. The number $\w_x$ models the energy stored at site $x$ and 
is called the {\sl weight} or {\sl potential} or {\sl environment} at $x$.
Define the natural {\sl shift maps} $T_z:\Omega_0\to\Omega_0$, $z\in\Z^2$, by 
	$\w_x(T_z \w) = \w_{x+z}(\w)$. 
We are given a probability measure $\P_0$ on $(\Omega_0,\cF_0)$ such that $\{\w_x:x\in\Z^2\}$ are i.i.d.\ under $\P_0$ and  
$\E_0[\abs{\w_0}]<\infty$.

Let $\Pi_{u,v}$ be the set of up-right paths, i.e.\ paths in $\Z^2$ with steps in $\{e_1,e_2\}$, from $u$ to $v$. 
For $m\le n$ in $\Z\cup\{\pm\infty\}$ we write $x_{m,n}$ to denote a path $(x_m,x_{m-1},\dotsc,x_n)$ and we will use the convention that $x_k\cdot(e_1+e_2)=k$.

Given the weights, the {\sl quenched point-to-point polymer measures} 
are probability measures on up-right paths between two fixed sites in which the probability of a path is proportional to the exponential of its energy:
	\begin{align}\label{p2p}
	Q_{u,v}^\w(x_{m,n})=\frac{e^{\sum_{k=m+1}^{n}\w_{x_k}}}{Z_{u,v}^\w}\,,\quad x_{m,n}\in\Pi_{u,v},v\ge u.
	\end{align}
Here, $Z_{u,v}^\w$ is the {\sl quenched point-to-point partition function} given by
	\begin{align}\label{def:Z}
	Z_{u,v}^\w=\sum_{x_{m,n}\in\Pi_{u,v}}e^{\sum_{k=m+1}^{n}\w_{x_k}},\quad{v\ge u},
	\end{align}
with the convention that an empty sum equals $0$.  (Thus, $Z_{u,u}^\w=1$ and $Z_{u,v}^\w=0$ when $v\not\ge u$.)  

A computation shows that the point-to-point measure is a backward Markov 
chain starting at $v$, taking steps $\{-e_1,-e_2\}$, and with 
absorption at $u$. If we define
	\[B_u(x,y,\w)=\log Z^\w_{u,y}-\log Z^\w_{u,x},\quad x,y\ge u,\]
then the transition probabilities of this Markov chain are given by
	\begin{align}\label{eq:RWREu}
	\cev\pi_u(x,x-e_i,\w)=\frac{e^{\w_x}Z_{u,x-e_i}^\w}{Z_{u,x}^\w}
	=e^{\w_x-B_u(x-e_i,x,\w)}\,,
	\  u\le x\le v, x\ne u, i\in\{1,2\}.
	\end{align}
Note that if $x=u+ke_i$, $k\in\N$, then the chain takes $-e_i$ steps until it reaches $u$. The processes $B_u$ satisfy the (rooted) {\sl cocycle property}
	\begin{align}\label{Bu-coc}
	&B_u(x,y,\w)+B_u(y,z,\w)=B_u(x,z,\w),\quad x,y,z\ge u.
	\end{align}
Since $Z^\w_{u,v}$ satisfies the recurrence
	\begin{align}\label{Zrec}
	Z_{u,x}=e^{\w_x}(Z_{u,x-e_1}+Z_{u,x-e_2}),\quad x-u\in\N^2,
	\end{align}
we see that $B_u$ also satisfies the {\sl recovery property}
	\begin{align}\label{Bu-rec}
	&e^{-B_u(x-e_1,x,\w)}+e^{-B_u(x-e_2,x,\w)}=e^{-\w_x},\quad x-u\in\N^2.
	\end{align}
Note also that 
	\begin{align}\label{Bu-cov}
	B_u(x,y,T_z\w)=B_{u+z}(x+z,y+z,\w)
	\end{align} 
and that  if $x,y\le v$, then $B_u(x,y,\w)$ is a function of $\{\w_z:z\le v\}$ and is hence independent of $\{\w_z:z\not\le v\}$.
\smallskip

A stationary polymer measure is one that retains the properties \eqref{Bu-coc}, \eqref{Bu-rec}, \eqref{Bu-cov}, 
and the  above independence structure, but without a reference to the roots $u$ and $u+z$. This leads us to the notion of corrector distributions. 

\subsection{Corrector distributions}
Extend the measurable space $(\Omega_0,\cF_0)$ to $(\Omega,\cF)$ where
$\Omega=\R^{\Z^2}\times\R^{\Z^2\times\Z^2}$, equipped with the product topology, and $\cF$ is the product Borel 
$\sigma$-algebra. Now, $\w$ will denote a generic point in $\Omega$  
and $\{\w_x(\w):x\in\Z^2\}$ and $\{B(x,y,\w):x,y\in\Z^2\}$ denote the natural coordinate projections. 
The natural shift maps are now given by $T_z:\Omega\to\Omega$, $z\in\Z^2$, with
	$\w_x(T_z \w) = \w_{x+z}(\w)$ and $B(x,y,T_z\w)=B(x+z,y+z,\w)$. 
We will abuse notation and keep using $\w_x$ and $T_z$ to denote, respectively, the natural projections and shifts on the original space $\Omega_0$.\smallskip

We say that a probability measure $\P$ on $(\Omega,\cF)$ is a {\sl stationary future-independent $L^1$ corrector distribution with $\Omega_0$-marginal $\P_0$} if it satisfies the following:\smallskip

\noindent{\rm I. Distributional properties:} for all $x,y,z\in\Z^2$
\begin{enumerate}[label={\rm(\alph*)}, ref={\rm\alph*}, topsep=2pt, partopsep=1pt, itemsep=0.5pt] 
	\item Prescribed marginal: the $\Omega_0$-marginal is $\P_0$,
	\item Stationarity: $\P$ is invariant under $T_z$,
	\item Integrability: $\E[\abs{B(x,y)}]<\infty$,
	\item\label{future} Future-independence: for any down-right path $\bfy=(y_k)_{k\in\Z}$, i.e.\  $y_{k+1}-y_k\in\{e_1,-e_2\}$, 
		$\{B(x,y,\w):\exists v\in\bfy:x,y\le v\}$ 
		and $\{\w_z:z\not\le v,\forall v\in\bfy\}$ 
		are independent.
	\end{enumerate}
\noindent{\rm II. Almost sure properties:}  for $\P$-almost every $\w$ and all $x,y,z\in\Z^2$
	\begin{enumerate}[resume,label={\rm(\alph*)}, ref={\rm\alph*}, topsep=2pt, partopsep=1pt, itemsep=0.5pt]
	\item\label{cocycle} Cocycle: $B(x,y)+B(y,z)=B(x,z)$,
	\item\label{recovery} Recovery: $e^{-B(x-e_1,x)}+e^{-B(x-e_2,x)}=e^{-\w_x}$.
	\end{enumerate}\smallskip
	
We say $\P$ is {\sl ergodic} under $T_z$ (or $T_z$-ergodic) if $\P(A)\in\{0,1\}$ for all sets $A\in\cF$ such that $T_zA =A$. 	
As it is customary with probability notation (and was already done above), we will often omit the $\w$ from the arguments of $B(x,y)$ and $\w_x$.
A function $B$ satisfying property \eqref{cocycle} is called a {\sl cocyle}. If it also satisfies  \eqref{recovery} then it is called a {\sl corrector}.
Our use of the word ``corrector'' comes from an analogy with stochastic homogenization. 
See e.g.\ \cite[page 467]{Arm-Sou-12}. The recovery equation \eqref{recovery} is the analogue of  (3.4) in that paper.

\subsection{Stationary polymer measures from corrector distributions}\label{sec:statpol}

A {\sl stationary polymer measure} is given by first specifying a stationary future-independent corrector distribution $\P$ with an i.i.d.\ $\Omega_0$-marginal $\P_0$. Given a realization of the environment $\w$, the {\sl quenched polymer measure} $Q_v^\w$ rooted at $v\in\Z^2$ is a Markov chain starting at $v$ and having transition probabilities
	\begin{align}\label{RWRE}
	\cev\pi(x,x-e_i,\w)=e^{\w_x-B(x-e_i,x,\w)},\quad x\in\Z^2,i\in\{1,2\}.
	\end{align}
Observe that $\cev\pi(x,x-e_i,\w)=\cev\pi(0,-e_i,T_x\w)$. Hence,
stationary polymer measures are in fact examples of the familiar model of a {\sl random walk in a stationary random environment} (RWRE). 

The quenched point-to-point measure \eqref{p2p} can also be viewed as a forward Markov chain starting at $u$, taking steps $\{e_1,e_2\}$, with absorption at $v$ and transitions
	\begin{align}\label{RWRE-forward}
	\vec\pi_v(x,x+e_i,\w)=\frac{e^{\w_{x+e_i}}Z_{x+e_i,v}^\w}{Z_{x,v}^\w}
	=e^{\w_x-\Bbar_v(x,x+e_i)}\,,
	\ u\le x\le v, x\ne v, i\in\{1,2\},
	\end{align}
where now $\Bbar_v(x,y,\w)=\log Z^\w_{x,v}-\log Z^\w_{y,v}+\w_x-\w_y$. This structure also leads to stationary polymer measures that are stationary (forward) RWREs with steps $\{e_1,e_2\}$ and whose transitions are of the form $\vec\pi(x,x+e_i,\w)= e^{\w_x-\Bbar(x,x+e_i,\w)}$, $i\in\{1,2\}$, where $\Bbar$ is an $L^1$ stationary corrector but with the recovery equation replaced by $e^{-\w_x}=e^{-\Bbar(x,x+e_1)}+e^{-\Bbar(x,x+e_2)}$ and future-independence replaced by past-independence (defined in the obvious way). The two points of view are in fact equivalent due to the symmetry of $\P_0$ with respect to reflections of the axes.\smallskip

It should be noted that although the weights $\{\w_x:x\in\Z^2\}$ are i.i.d.\ under $\P_0$, the transitions $\{\cev\pi(x,x-e_1):x\in\Z^2\}$ (and $\{\vec\pi(x,x+e_1):x\in\Z^2\}$) 
are highly correlated, causing the paths of the RWREs to be superdiffusive 
with a $2/3$ scaling exponent. 
See for example Theorem 7.2 of \cite{Geo-etal-15}.

\subsection{Stationary polymer measures with boundary}\label{sec:bdry}

Another, perhaps more familiar, way of introducing stationary polymer measures comes by considering solutions to the recursion \eqref{Zrec}, but with appropriate boundary conditions. 
This is how these measures were introduced in the study of solvable models mentioned in the introduction. 
We explain in this section how this viewpoint is the same as the one via the framework of corrector distributions.\smallskip

Given a stationary future-independent corrector distribution $\P$ with an i.i.d.\ $\Omega_0$-marginal $\P_0$,  a down-right path $\bfy=y_{-\infty,\infty}$ with $y_m\cdot(e_1-e_2)=m$ for $m\in\Z$, and a point $u\in\bfy$, 
define the {\sl quenched path-to-point partition functions}
	\begin{align}\label{Zpath2pt}
	Z_{u,v}^{\bfy,\w}=\sum_{x_{m,n}\in\Pi_{\bfy,v}}e^{B(u,x_m)+\sum_{k=m+1}^{n}\w_{x_k}}.
	\end{align}
Here, $\Pi_{\bfy,v}$ is the set of up-right paths $x_{m,n}$ that start at a point $x_m\in\bfy$, exit $\bfy$ right away, i.e.\ $x_{m+1}\not\in\bfy$, and end at $x_n=v$. Recall that an empty sum is $0$. If $v\in\bfy$, then $\Pi_{\bfy,v}$ consists of a single path and $Z_{u,v}^{\bfy,\w}=e^{B(u,v)}$.   

The cocycle and recovery properties \eqref{cocycle} and \eqref{recovery} imply that $e^{B(u,x)}$ satisfies the same recurrence relation \eqref{Zrec} as $Z^{\bfy,\w}_{u,x}$. 
Since the two also match for $x\in\bfy$ we deduce that $Z^{\bfy,\w}_{u,x}=e^{B(u,x)}$ for all $x$ for which $\Pi_{\bfy,x}\not=\varnothing$.  
In particular, this definition is independent of the boundary $\bfy$ and gives a
stationary field $\{e^{B(u,v)}:u,v\in\Z^2\}$ of point-to-point partition functions.
Also, this explains why $B$ is called a corrector: it corrects the potential $\{\w_x:x\in\Z^2\}$, turning the superadditive $\log Z_{u,v}^\w$ into an additive cocycle $\log Z_{u,v}^{\bfy,\w}$. This is a key idea in stochastic homogenization theory.   For more, see for example Section 2 of \cite{Kos-07}.\smallskip

The corresponding {\sl quenched path-to-point polymer measure} is  given by
	\[Q_{u,v}^{\bfy,\w}(x_{m,n})=\frac{e^{B(u,x_m) + \sum_{k=m+1}^{n}\w_{x_k}}}{Z_{u,v}^{\bfy,\w}}\,,\quad x_{m,n}\in\Pi_{\bfy,v}.\]
$Q_{u,v}^{\bfy,\w}$ is the distribution of a Markov chain starting at $v$ and having transition probabilities 
	\[\cev\pi(x,x-e_i,\w)=\frac{e^{\w_x}Z^{\bfy,\w}_{u,x-e_i}}{Z^{\bfy,\w}_{u,x}}=e^{\w_x-B(x-e_i,x,\w)},\quad x\in\Z^2,i\in\{1,2\}\]
until reaching $\bfy$.  In other words, this is exactly the quenched distribution $Q^\w_v$, until absorption at $\bfy$ and the path-to-point polymer measure is exactly the stationary polymer measure introduced above.\smallskip

One can also go in the other direction: starting from a stationary model with boundary we can define a corrector distribution. More precisely, suppose we are given a boundary down-right path 
$\bfy=y_{-\infty,\infty}$ with $y_m\cdot(e_1-e_2)=m$ for $m\in\Z$ and a point $u\in\bfy$. Abbreviate 
$\bbI_{\bfy}^+=\{z\in\Z^2:z\not\le v,\forall v\in\bfy\}$. Equip  $\Omega_{\bfy}=\R^{\bbI_{\bfy}^+}\times\R^{\bfy}$ with the product topology and Borel $\sigma$-algebra
and denote the natural projections of an element $\w\in\Omega_{\bfy}$ by $\w_z$, $z\in\bbI_{\bfy}^+$, and 
$\bar\w_v$, $v\in\bfy$.  Suppose we are given a  
probability measure $\P_1$ on $\Omega_{\bfy}$ under which
$\{\w_z:z\in\bbI_{\bfy}^+\}$ and $\{\bar\w_v:v\in\bfy\}$ are independent and such that
the distribution of $\{\w_z:z\in\bbI_{\bfy}^+\}$ is the same under $\P_1$ as under $\P_0$. 

Let $m_0=u\cdot(e_1-e_2)$, so that $u=y_{m_0}$.  For $m\in\Z$ let $B(u,x_m)=\sum_{i=m_0}^{m-1}\bar\w_i$ for $m\ge m_0$ and $B(u,x_m)=-\sum_{i=m}^{m_0-1}\bar\w_i$ for $m\le m_0$.
Define the path-to-point partition function $Z_{u,v}^{\bfy,\w}$, $v\in\bbI_{\bfy}^+\cup\bfy$,
by \eqref{Zpath2pt}. The probability measure $\P_1$ is said to be a {\sl stationary polymer model with boundary} $\bfy$ if the distribution of 
\[\{\w_{v+z},Z_{u,y+z}^{\bfy,\w}/Z_{u,u+z}^{\bfy,\w}:v\in\bbI_{\bfy}^+,\,y\in \bbI_{\bfy}^+\cup\bfy\},\] 
induced by $\P_1$, does not depend on $z\in\Z_+^2$.

For $x,y\in\bbI_{\bfy}^+\cup\bfy$ let
	\[B_{\bfy}(x,y)=\log Z_{u,y}^{\bfy,\w}-\log Z_{u,x}^{\bfy,\w}.\]
Then the above definition is equivalent to saying that the distribution of 
\[\{\w_{v+z},B_{\bfy}(x+z,y+z):v\in\bbI_{\bfy}^+,\,x,y\in\bbI_{\bfy}^+\cup\bfy\},\] 
induced by $\P_1$, is the same for all $z\in\Z_+^2$. Kolmogorov's consistency theorem allows then to extend $\P_1$ to a probability measure $\P$ on 
$(\Omega,\cF)$ and a few direct computations check that $\P$ is 
a stationary future-independent corrector distribution with $\Omega_0$-marginal $\P_0$.

\subsection{The stationary log-gamma polymer}\label{sec:boundary}

As mentioned earlier, stationary polymer measures with boundaries were a crucial tool in the study of solvable models. 
In this section we recall the example of Sepp\"al\"ainen's log-gamma polymer 
\cite{Sep-12-corr}, which fits our setting. 
Related models include the stationary semi-discrete model in \cite{Oco-Yor-01}, where 
the boundary $(-\infty,\infty) \times \{0\}$ was used, which is analogous to $y_{-\infty,\infty}=\Z e_1$,
and the models studied in \cite{Bar-Cor-17, Cor-Sep-She-15, Thi-Dou-15}, where 
$y_{-\infty,0} = \Z_+ e_2$ and $y_{0,\infty}=\Z_+ e_1$ was used. In all of these models, the reference point is taken to be $u=0$. \smallskip

For $\theta>0$ let $\bbW_\theta$ denote the distribution of a random variable $X$ such that
$e^{-X}$ is gamma-distributed with scale parameter $1$ and shape parameter $\theta$. 
Let $\overline\bbW_\theta$ denote the distribution of $-X$.
The log-gamma polymer is the directed polymer measure on $\Z^2$ with $\P_0$ being the product measure $\bbW_\rho^{\otimes\Z^2}$ for some $\rho>0$.

Consider the boundary path $y_{-\infty,0}=\Z_+ e_2$ and $y_{0,\infty}=\Z_+ e_1$ and the origin point $u=0$. Then $\bbI_{\bfy}^+=\N^2$.
Fix $\theta\in(0,\rho)$ and let $\P_1$ be the product probability measure 
$\bbW_\rho^{\otimes\N^2}\otimes\bbW_\theta^{\otimes{\N e_1}}\otimes\overline\bbW_{\rho-\theta}^{\otimes{\N e_2}}$.  
The path-to-point partition functions $Z_{0,x}^{\bfy,\w}$, $x\in\Z_+^2$, can be computed inductively by the equations
	\[Z_{0,x}^{\bfy,\w}=e^{\w_x}(Z_{0,x-e_1}^{\bfy,\w}+Z_{0,x-e_2}^{\bfy,\w}),\quad x\in\N^2,\]
and the initial conditions 
	\[Z_{0,0}^{\bfy,\w} = 1, \quad Z_{0,me_1}^{\bfy,\w}=e^{\sum_{i=0}^{m-1}\bar\w_{ie_1}}\quad\text{and}\quad
	Z_{0,me_2}^{\bfy,\w}=e^{-\sum_{i=1}^{m}\bar\w_{ie_2}},\quad m\in\Z_+.\]
Equivalently, $B_{\bfy}(x,x+e_i)$, $i\in\{1,2\}$, $x\in\Z_+^2$, are computed inductively 
using
	\begin{align}\label{B-Lindley}
	\begin{split}
	&e^{B_{\bfy}(x-e_1,x)}=e^{\w_x}\bigl(1+e^{B_{\bfy}(x-e_1-e_2,x-e_2)}e^{-B_{\bfy}(x-e_1-e_2,x-e_1)}\bigr),\\
	&e^{B_{\bfy}(x-e_2,x)}=e^{\w_x}\bigl(1+e^{B_{\bfy}(x-e_1-e_2,x-e_1)}e^{-B_{\bfy}(x-e_1-e_2,x-e_2)}\bigr),
	\end{split}
	\end{align}
for $x\in\N^2$, and the initial conditions $B_{\bfy}(me_1,(m+1)e_1)=\bar\w_{me_1}$ and 
$B_{\bfy}(me_2,(m+1)e_2)=-\bar\w_{(m+1)e_2}$, $m\in\Z_+$. 
Compare to (3.2) in \cite{Sep-12-corr}.
Then $B_{\bfy}(x,y)$, $x,y\in\Z_+^2$, are computed via the cocycle property that $B_{\bfy}$ satisfies.
The Burke property \cite[Theorem 3.3]{Sep-12-corr} implies that $\P_1$ is stationary in the sense of  the previous section.\smallskip

Alternatively, 
one can use the boundary path $y_{-\infty,0}=\Z_+e_2$ and $y_{0,\infty}=\Z_+ e_1$ and then 
$\bbI_{\bfy}^+=\Z\times\N$ and $\P_1$ would be the product measure
$\bbW_\rho^{\otimes\bbI_{\bfy}^+}\otimes\bbW_\theta^{\otimes\bfy}$. The partition functions
$Z_{0,x}^{\bfy,\w}$, $x\in\Z\times\Z_+$, are now computed inductively by the equations
\[Z_{0,x}^{\bfy,\w}=\sum_{m=0}^\infty e^{\sum_{i=0}^m\w_{x-ie_1}}Z_{0,x-e_2-me_1}^{\bfy,\w},\quad x\in\Z\times\N,\]
and the initial conditions 
	\[Z_{0,0}^{\bfy,\w} = 1, \quad Z_{0,me_1}^{\bfy,\w}=e^{\sum_{i=0}^{m-1}\bar\w_{ie_1}}\quad\text{and}\quad
	Z_{0,-me_1}^{\bfy,\w}=e^{-\sum_{i=1}^{m}\bar\w_{-ie_1}},\quad m\in\Z_+.\]
Equivalently, $B_{\bfy}(x,x+e_i)$, $i\in\{1,2\}$, $x\in\Z\times\Z_+$ are computed inductively 
using
	\begin{align}
	&e^{B_{\bfy}(x-e_2,x)}=e^{\w_x}\Bigl(1+\sum_{m=1}^\infty\prod_{i=1}^m e^{\w_{x-ie_1}-B_{\bfy}(x-e_2-ie_1,x-e_2-(i-1)e_1)}\Bigr),\label{LindB}\\
	&e^{B_{\bfy}(x-e_1,x)}=e^{\w_x}\bigl(1+e^{B_{\bfy}(x-e_1-e_2,x-e_2)}e^{-B_{\bfy}(x-e_1-e_2,x-e_1)}\bigr),\notag
	\end{align}
for $x\in\Z\times\N$, and the initial conditions $B_{\bfy}(me_1,(m+1)e_1)=\bar\w_{me_1}$, $m\in\Z$. Analysis of the general polymer model with boundary conditions of this type plays a key role in our analysis. See in particular the discussion in Section \ref{sec:aux}.

By the uniqueness in Lemma \ref{lem:Yunique}, the distribution of $\{B_{\bfy}(x,x+e_i):x\in\Z_+^2\}$ in this construction is the same as the one in the above construction. Consequently, $\P_1$ is again stationary in the sense of the previous section. See Lemma \ref{lm:lineBurke} for the details of this argument. 

\section{Main results}\label{sec:results}

Consider a stationary future-independent $L^1$ corrector distribution $\P$ with $\Omega_0$-marginal $\P_0$.
By shift-invariance and the cocycle property, 
	\[\E[B(0,x+y)]=\E[B(0,x)]+\E[B(x,x+y)]=\E[B(0,x)]+\E[B(0,y)],\] 
for all $x,y\in\Z^2$. Hence, there exists a unique vector $\bfm_\P\in\R^2$, called the {\sl mean vector}, such that $\E[B(0,x)]=x\cdot\bfm_\P$ for all $x\in\Z^2$. In particular, $\bfm_\P\cdot e_i=\E[B(0,e_i)]$.\smallskip


Our first main result is on the uniqueness of ergodic corrector distributions with prescribed i.i.d.\ $\Omega_0$-marginal and mean $\bfm_\P\cdot e_i$.


\begin{theorem}\label{thm:uniqueness}
Fix an i.i.d.\ probability measure $\P_0$ on $(\Omega_0,\cF_0)$ with $L^1$ weights.
Fix a number $\alpha\in\R$. Fix $i\in\{1,2\}$. There is at most one stationary $T_{e_i}$-ergodic future-independent $L^1$ corrector distribution $\P$ with $\Omega_0$-marginal $\P_0$ and such that $\bfm_\P\cdot e_i=\alpha$. 
\end{theorem}

In terms of stationary polymers with boundary $\bfy=\bbZ e_1$, 
this theorem says that, in general, the $T_{e_1}$-ergodic stationary distributions form a one parameter family, indexed by the mean of the boundary weights.
The formulation in terms of corrector distributions allows us to extend this result to more general boundary geometries.\smallskip
 
We next turn to the question of determining which values of the parameter $\bfm_\P$ admit ergodic stationary polymers. To state our second result we need a few more definitions and some more hypotheses. 
Recall the point-to-point partition functions \eqref{def:Z}.
Assume that $\E[\abs{\w_0}^p]<\infty$ for some $p>2$. Then Theorem 2.2(a), Remark 2.3, and Theorems 2.4, 2.6(b), and  3.2(a) of \cite{Ras-Sep-14} imply that there exists a deterministic continuous 
$1$-homogenous concave function $\fe_{\P_0}:\R_+^2\to\R$ such that $\P_0$-almost surely
	\begin{align}\label{free energy}
	n^{-1}\log Z_{0,\fl{n\xi}}^\w\mathop{\longrightarrow}_{n\to\infty}\fe_{\P_0}(\xi)\quad\text{for all $\xi\in\R_+^2$.}
	\end{align}

Given $\xi\in(0,\infty)^2$, let 
	\begin{align}\label{super diff}
	\partial\fe_{\P_0}(\xi)=\bigl\{\bfb\in\R^2: \bfb\cdot(\xi-\zeta)\le \fe_{\P_0}(\xi)-\fe_{\P_0}(\zeta),\ \forall\zeta\in\R_+^2\bigr\}
	\end{align}
denote the superdifferential of $\fe_{\P_0}$ at $\xi$. This is a convex set.
Let $\ext\partial\fe_{\P_0}(\xi)$ denote its extreme points. If $\xi\in(0,\infty)^2$ then $\fe_{\P_0}$ is differentiable at $\xi$ if and only if  
	\begin{align}\label{extremes}
	\partial\fe_{\P_0}(\xi)=\ext\partial\fe_{\P_0}(\xi)=\{\nabla\fe_{\P_0}(\xi)\}.
	\end{align}
Otherwise, $\ext\partial\fe_{\P_0}(\xi)$ consists of exactly two points (see Lemma 4.6(c) in \cite{Jan-Ras-18-arxiv}). 
It is conjectured that $\fe_{\P_0}$ is differentiable on $(0,\infty)^2$ and then \eqref{extremes} holds 
for all $\xi\in(0,\infty)^2$. Note that due to the homogeneity of $\fe_{\P_0}$, $\partial\fe_{\P_0}(\xi)=\partial\fe_{\P_0}(c\xi)$ for all $\xi\in(0,\infty)^2$ and $c>0$.\smallskip

Before presenting our second main result in the present paper, we record some useful inputs from our companion paper \cite{Jan-Ras-18-arxiv}. The first is Lemma 4.5(a) in that paper and it characterizes the possible mean vectors of corrector distributions.
\begin{lemma}
If $\P$ is a stationary future-independent corrector distribution with an $\Omega_0$-marginal given by i.i.d.\ $L^p$ weights, $p>2$, then $\bfm_\P\in\partial\fe_{\P_0}(\xi)$ for some $\xi \in \{t e_1 + (1-t) e_2 : 0 < t < 1\}=]e_1,e_2[$.
\end{lemma}
The second is Theorem 4.7 in that paper and gives existence of corrector distributions for each such mean.

\begin{lemma}\label{lm:thm4.7}
For each $\bfb$  with the property that $\bfb\in\partial\fe_{\P_0}(\xi)$ for some $\xi \in]e_1,e_2[$ and for each probability measure $\P_0$ on $\Omega_0$ under which the weights are i.i.d. and in $L^p$ for some $p>2$, there exists a stationary future-independent $L^1$ corrector distribution $\P$ with $\Omega_0$-marginal $\P_0$ such that $\bfm_\P=\bfb$. 
\end{lemma}
Let $\sC_{\P_0}$ denote the collection of all stationary future-independent $L^1$ corrector distributions with $\Omega_0$-marginal $\P_0$. 
In words, the last lemma says that as $\P$ varies over $\sC_{\P_0}$,  its mean vector $\bfm_\P$ spans all of $\bigcup_{\xi\in]e_1,e_2[}\partial\fe_{\P_0}(\xi)$. 
The next result is an immediate consequence of Lemmas 4.7(b), 4.7(d), and C.1 in \cite{Jan-Ras-18-arxiv}. It 
says that in fact as $\P$ spans $\sC_{\P_0}$ each coordinate of $\bfm_\P$ 
spans $(\E_0[\w_0],\infty)$.

\begin{lemma}\label{lm:span}
Fix an i.i.d.\ probability measure $\P_0$ on $(\Omega_0,\cF_0)$ with $L^p$ weights, for some $p>2$. Then $\{\bfm_\P: \P\in \sC_{\P_0}\}$ is a closed curve in $\R^2$ and  for each $i \in \{1,2\}$,  $\{\bfm_\P \cdot e_i : \P \in \sC_{\P_0}\} = (\E_0[\w_0],\infty)$.
\end{lemma}

\smallskip

Our second main result in this paper gives a convenient tool which allows us to identify ergodic stationary corrector distributions.
\begin{theorem}\label{thm:ergodicity}
Fix an i.i.d.\ probability measure $\P_0$ on $(\Omega_0,\cF_0)$ with $L^p$ weights, for some $p>2$.
Suppose  $\P$ is a stationary future-independent corrector distribution with 
$\Omega_0$-marginal $\P_0$. If $\bfm_\P\in\ext \partial \fe_{\P_0}(\xi)$ for  some $\xi\in]e_1,e_2[$, then
$\P$ is ergodic under  $T_{e_1}$ and $T_{e_2}$.
\end{theorem}

In principle, this result leaves open the possibility that there could be ergodic stationary distributions with mean vectors which are not extreme points of the superdifferential of the free energy. Nevertheless, in our setting, it is expected that $\fe_{\P_0}$ is differentiable, in which case every element of the superdifferential would be extreme. In particular, under this hypothesis, 
for each $\xi\in]e_1,e_2[$, Lemma \ref{lm:thm4.7} furnishes a future-independent $L^1$ corrector distribution $\P_\xi$ with $\Omega_0$-marginal $\P_0$ such that $\bfm_\P=\nabla\fe_{\P_0}(\xi)$. Then, under the hypothesis of differentiability, using Theorems \ref{thm:uniqueness} and \ref{thm:ergodicity} we have the following complete characterization of all ergodic stationary polymer meausures.

\begin{corollary}
Fix an i.i.d.\ probability measure $\P_0$ on $(\Omega_0,\cF_0)$ with $L^p$ weights, for some $p>2$. Assume $\fe_{\P_0}$ is differentiable on $(0,\infty)^2$. Then for $i \in \{1,2\}$, the collection of $T_{e_i}$-ergodic stationary corrector distributions is exactly given by $\{\P_\xi : \xi \in ]e_1,e_2[\}$. In particular, for each $\xi\in]e_1,e_2[$ and each $i\in\{1,2\}$, $\P_\xi$ is the unique $T_{e_i}$-ergodic future-independent corrector distribution with $\Omega_0$-marginal $\P_0$ and such that $\bfm_\P=\nabla\fe_{\P_0}(\xi)$.
\end{corollary}

 As a consequence, the above and Lemma \ref{lm:span} imply that the one-parameter family of ergodic 
 measures constructed in \cite{Sep-12-corr} is unique
 and assumption (2-6) of the scaling theory in \cite{Spo-14} is satisfied.

\section{Preliminaries}\label{sec:prelims}
In this section, we motivate the main tool in the proofs of the main 
theorems \ref{thm:uniqueness} and \ref{thm:ergodicity}, which we will call the {\sl update map} $\Phi$. 
\smallskip 

Given $B:\Z^2\times\Z^2\to\R$ define the random variables $V_{n,k}=e^{\w_{(n,k)}}$, $X_{n,k}=e^{B((n,k-1),(n+1,k-1))}$, and $Y_{n,k}=e^{B((n,k-1),(n,k))-\w_{(n,k)}}$
for $n,k\in\Z$. The next lemma rewrites the corrector property in terms of this notation.  Compare \eqref{Lindley} with \eqref{B-Lindley}.

\begin{lemma}\label{lm:Lindley}
$B$ is a corrector if and only if the following hold for all $n,k\in\Z$:
\begin{align}\label{Lindley}
Y_{n+1,k} = 1 + \frac{V_{n,k}}{X_{n,k}}Y_{n,k}\quad\text{and}\quad
X_{n,k+1} = V_{n+1,k}\Bigl(1 + \frac{X_{n,k}}{V_{n,k}Y_{n,k}}\Bigr).
\end{align}
\end{lemma}

\begin{proof}
The cocycle property \eqref{cocycle} is equivalent to 
\begin{align*}
B(x-e_2,x) - B(x-e_1,x)= B(x-e_1-e_2,x-e_1)-B(x-e_1-e_2,x-e_2)
\end{align*}
for all $x\in\Z^2$.
Together, the cocycle and the recovery properties \eqref{cocycle} and \eqref{recovery} are equivalent to
\begin{align*}
&e^{B(x-e_2,x)-\w_x} 
=1+e^{B(x-e_1-e_2,x-e_1)-\w_{x-e_1}}e^{\w_{x-e_1}}e^{-B(x-e_1-e_2,x-e_2)}\quad\text{and}\\
&e^{B(x-e_1,x)} =e^{\w_x}\Bigl(1+ e^{B(x-e_1-e_2,x-e_2)-\w_{x-e_2}}e^{\w_{x-e_2}}e^{-B(x-e_1-e_2,x-e_1)}\Bigr)
\end{align*}
holding for all $x\in\Z^2$.
Plug in $x=(n+1,k)$ in the first equation and $x=(n,k)$ in the second one, then apply the definitions of $V$, $X$, and $Y$.
\end{proof}

We will see below that iterating the first equation in \eqref{Lindley} gives
\begin{align}\label{Y-recursion}
Y_{n+1,k} = 1 + \sum_{j=-\infty}^{n} \prod_{i=j}^{n} \frac{V_{i,k}}{X_{i,k}}\quad\text{for all }n,k\in\Z.
\end{align}
Compare with \eqref{LindB}.

Suppose now $\{\w_x:x\in\Z^2\}$ have an i.i.d.\ distribution $\P_0$. Then $\{V_{n,k}:n,k\in\Z\}$ are also i.i.d.  Suppose  $\{X_{n,0}:n\in\Z\}$ are independent of the $V$ variables and have a stationary probability distribution $\mu$. 
Once the variables $\{V_{n,k}:n,k\in\Z\}$ and $\{X_{n,0}:n\in\Z\}$ are known, the rest of the variables $\{Y_{n,k}:n\in\Z,k\in\Z_+\}$ and $\{X_{n,k}:n\in\Z,k\in\N\}$ can be computed via \eqref{Y-recursion} and the second equation in \eqref{Lindley}.
The resulting process $\{V_{n,k},X_{n,k},Y_{n,k}:n\in\Z,k\in\Z_+\}$ is clearly stationary under shifts in the $n$ index.    

Let $\Phi(\mu)$ be the distribution of $\{X_{n,1}:n\in\Z\}$.
By the equivalence of stationary corrector distributions and stationary polymer measures with boundary, discussed in Section \ref{sec:bdry} 
(with here $\bfy=\Z e_1$),
the problem of finding a stationary corrector distribution $\P$ with $\Omega_0$-marginal $\P_0$ is  the same as finding a stationary fixed point for the map $\Phi$.
Indeed, this will ensure that the $(V,X,Y)$ process is stationary under shifts in the $k$ index. 
Inspecting the dependence on $V$ in \eqref{Lindley} and \eqref{Y-recursion} one can quickly see that $\P$ will also satisfy the future-independence property.

The purpose of the next section is to show that $\Phi$ is a mean-preserving contraction and hence has at most one fixed point with a given mean.


\section{Properties of the update map}\label{sec:aux}
%
%
%

The properties of the update map $\Phi$ that we prove here require fewer assumptions than our main results. Hence, this section has its own setting and notation, which we introduce next.\smallskip  

We are given a stationary process $X=\{X_n : n \in \Z\}$ and an i.i.d.\ sequence $V=\{V_n : n \in \Z\}$ with $X_0 > 0$ and  
$V_0 > 0$ almost surely. Assume the two families are independent of each other and denote their joint distribution by $P$, with expectation $E$. Suppose $E[|\log X_0 |] < \infty$ and $E[|\log V_0|] < \infty$. Let $\sI$ denote the shift-invariant $\sigma$-algebra of the process $\{(X_n, V_n) : n \in \Z\}$ and assume that 
	\begin{align}\label{eq:condineq}
	E[\log X_0 \,|\, \sI] > E[\log V_0]\quad P\text{-almost surely.}
	\end{align} 

\begin{lemma} \label{lem:cond0lim}
Suppose that $Y=\{Y_n:n\in\Z\}$  satisfies the recursion
\begin{align} \label{eq:Y0recursion}
Y_{n+1} &= 1 + V_nX_n^{-1}Y_n,\quad P\text{-almost surely and for all }n\in\Z.
\end{align}
Then $n^{-1} \one\{0<Y_0<\infty\}\log Y_n\to 0$ almost surely.
\end{lemma}

\begin{proof}
Let $a\in\R$ and $\e > 0$ be given and abbreviate $U_k = \log  V_k - \log  X_k$. Define
\begin{align*}
F_0^\e(a) = a\ \text{and}\ 
F^\e_{n+1}(a) =  \log\left(1 + \exp\left\{F_{n}^\e(a) +  U_n - E[U_0 \,|\, \sI]  + \e\right\}\right)
\quad\text{for $n\in\Z_+$}.
\end{align*}
An induction argument shows that for any $n \in \N$, we have 
\begin{align*}
F_n^\e(a) &= \log \Bigl(1 + \sum_{m=1}^{n-1} \exp\Bigl\{ \sum_{k=m}^{n-1}(  U_k - E[U_0 \,|\, \sI]  + \e)  \Bigr\}\\
&\qquad\qquad\qquad\qquad + \exp\Bigl\{a+\sum_{k=0}^{n-1} (U_k - E[U_0 \,|\, \sI]  + \e)  \Bigr\} \Bigr).
\end{align*}
As usual, we take an empty sum to be zero. 
The ergodic theorem implies then that $F_n^\e(a)=\log(1+e^{n\e+o(n)})$ and therefore $n^{-1}F_n^\e(a)\to\e$ $P$-almost surely, for any $a\in\R$.

Another induction  (using the fact that $E[U_0\,|\,\bar\sI]\le0$) shows that on the event $0<Y_0<\infty$, we have 
$0\le\log Y_n\le F_n^\e(\log Y_0)$ for all $n\in\N$. The claim of the lemma follows.
\end{proof}

\begin{lemma}\label{lem:Yunique}
The process $Y= \{Y_n : n \in \Z\}$ given by 
\begin{align} \label{eq:Ydef}
Y_n &= 1 + \sum_{m=-\infty}^{n-1} \prod_{k=m}^{n-1} \frac{V_k}{X_k}
\end{align}
is the unique stationary and almost surely finite solution to \eqref{eq:Y0recursion}.
For any other process $Y$ for which \eqref{eq:Y0recursion} holds $P$-almost surely  for all $n\in\Z$, $Y$ must satisfy  $P\bigl\{\lim_{j\to-\infty}\abs{Y_j} \to\infty \bigr\}> 0$ and then either $Y$ is stationary and $P(\abs{Y_{0}} = \infty) >0$ or $Y$ is not stationary.
\end{lemma}

\begin{proof}
Suppose $Y$ is given by \eqref{eq:Ydef}. 
By the ergodic theorem and \eqref{eq:condineq}
\begin{align}\label{eq:exp-decay}
\lim_{m\to-\infty} \frac{1}{\abs{m}} \sum_{k=m}^{n-1}\left( \log V_k - \log X_k\right) = E\left[ \log V_0 - \log X_0 \,|\, \sI \right] < 0\quad \text{almost surely.}
\end{align}
Hence $1 < Y_n < \infty$ almost surely. 
It is also clear that $Y$ is stationary and that \eqref{eq:Ydef} implies \eqref{eq:Y0recursion}. 

Conversely, let $Y$ satisfy \eqref{eq:Y0recursion} $P$-almost surely and for all $n\in\Z$. Iterating \eqref{eq:Y0recursion} implies that whenever $j < n-1$, we must have
\begin{align}\label{eq:rec-formula}
Y_{n} &= 1 + \sum_{m=j+1}^{n-1} \prod_{k=m}^{n-1} \frac{V_k}{X_k} + Y_{j} \prod_{k=j}^{n-1} \frac{V_{k}}{X_{k}}.
\end{align}

Suppose now that
\begin{align}\label{eq:case1}
P\Bigl\{\varliminf_{j\to - \infty} \abs{j}^{-1} \log \abs{Y_j} > 0\Bigl\} =0,
\end{align}
where we  take the convention that $\log \infty = \infty$.
Then this and \eqref{eq:exp-decay} imply that almost surely $\varliminf_{j\to-\infty} \abs{Y_j}\prod_{k=j}^{n-1} \frac{V_{k}}{X_{k}} = 0$. In this case,
taking $j\to-\infty$  in \eqref{eq:rec-formula} along a subsequence that realizes this liminf  implies $Y$ is given by \eqref{eq:Ydef} almost surely and for all $n\in\Z$.

If, alternatively, the probability in \eqref{eq:case1} is positive, then with positive probability
$\abs{Y_j}\to \infty$ as $j \to -\infty$. 
If we furthermore assume that $Y$ is stationary, then the ergodic theorem implies that $\abs{Y_0}=\infty$ on the event $\{\abs{Y_j}\to\infty\}$. 
To see this last claim note that for any $c>0$, 
$|k|^{-1}\sum_{i=k}^0\one\{|Y_k|\ge c\}\to P( |Y_0|\ge c \,|\, \sI )$ almost surely.
Consequently, 
	\begin{align*}
	P\bigl\{ |Y_0|\not=\infty, |Y_k|\to\infty \bigr\} 
	&= E\bigl[ \one\{|Y_0|\not=\infty\} \one\{|Y_k|\to\infty\} \bigr]\\
         &= E\bigl[ P(|Y_0|\not=\infty \,|\, \sI ) \one\{|Y_k|\to\infty\} \bigr] = 0.\qedhere
          \end{align*}
\end{proof}

Given the setting at the beginning of the section, define $Y=\{Y_n:n\in\Z\}$ by \eqref{eq:Ydef}. By Lemma \ref{lem:Yunique}, $Y$ satisfies \eqref{eq:Y0recursion} and $1<Y_0<\infty$ almost surely.  Define the stationary process 
\begin{align}
\Xbar_n &= V_{n+1}(1+ V_n^{-1}X_nY_n^{-1}) 
=  V_{n+1} Y_{n+1} \,\frac{X_n}{V_nY_n}\in(0,\infty)\,,\quad n\in\Z.  \label{eq:X1def1}
\end{align}
An induction argument shows that for $n\in\N$, we have
\begin{align} \label{eq:product}
Y_1V_1\prod_{k=1}^n \Xbar_k &= Y_{n+1}V_{n+1}\prod_{k=1}^n X_k.
\end{align}
Lemma \ref{lem:cond0lim} implies that $\log Y_n/n\to0$ almost surely. Also, $\log \Xbar_0>\log V_1$ almost surely. Hence, the ergodic theorem implies $\log\Xbar_0$ is integrable and
	\begin{align}\label{eq:means}
	E[\log \Xbar_0\, |\, \sI] = E[\log X_0\, |\, \sI].
	\end{align} 
\begin{lemma} \label{lem:mono}
Suppose $\{(X_n^1, V_n, Y_n^1,\Xbar^1_n) : n \in \Z\}$ and $\{(X_n^2, V_n, Y_n^2,\Xbar^2_n) : n \in \Z\}$ both satisfy equations \eqref{eq:Ydef} and \eqref{eq:X1def1}. Note that both families share the same $V$ variables. Suppose also that 
$X_n^1 \leq X_n^2$ for all $n\in\Z$. Then $Y_n^1 \geq Y_n^2$ and $\Xbar_n^1\leq \Xbar_n^2$ for all $n\in\Z$.
\end{lemma}

\begin{proof}
It follows immediately that if $X_n^1\leq X_n^2$ for all $n\in\Z$, then
\begin{align*}
Y_n^1 &= 1 + \sum_{m=-\infty}^{n-1} \prod_{k=m}^{n-1} \frac{V_k}{X_k^1} \geq 1 + \sum_{m=-\infty}^{n-1} \prod_{k=m}^{n-1} \frac{V_k}{X_k^2} = Y_n^2. 
\end{align*}
Then one has $X_n^1/Y_n^1 \leq X_n^2/Y_n^2$ for all $n\in\Z$. It follows that
\begin{align*}
\Xbar_n^1 &= V_{n+1}\Bigl(1+ \frac{X_{n}^1}{V_{n} Y_{n}^1 }\Bigr) \leq V_{n+1}\Bigl(1+ \frac{X_{n}^2}{V_{n} Y_{n}^2 }\Bigr) = \Xbar_{n}^2.\qedhere
\end{align*}
\end{proof}

Now we define the {\sl update operator} $\Phi:\cM_1(\R^\Z)\to\cM_1(\R^\Z)$, where $\cM_1(\X)$ is the set of probability measures on $\X$. 
Let $\Omega_A=\Omega_W=\R^\Z$. Let $A^0=(A_n^0)_{n\in\Z}$ and $(A^0,W)=(A^0_n,W_n)_{n\in\Z}$ be the natural coordinate projections on $\Omega_A$ and $\Omega_A\times\Omega_W$, respectively. 
Let $(A^1,A^2)=(A^1_n,A^2_n)_{n\in\Z}$ and $(A^1,A^1,W)=(A^1_n,A^2_n,W_n)_{n\in\Z}$ be the natural coordinate projections on $\Omega_A\times\Omega_A$ and $\Omega_A\times\Omega_A\times\Omega_W$, respectively. Equip all these spaces with the product topologies, Borel $\sigma$-algebras, and natural shifts. Let $\sI^0$ and $\sI$ be the $\sigma$-algebras of shift-invariant Borel subsets of $\Omega_A$ and $\Omega_A\times\Omega_W$, respectively.
Note that if we view $\Omega_A$ as embedded in $\Omega_A\times\Omega_W$ and abuse notation by identifying $\sI$ and $\sI\times\Omega_W$, then $\sI \subset\sI^0$. 

For $i\in\{0,1,2\}$, given numbers $(A^i_n)_{n\in\Z}$ and $(W_n)_{n\in\Z}$
define $X^i_n=e^{A^i_n}$ and $V_n=e^{W_n}$ then
$Y^i_n$ and $\Xbar^i_n$ via \eqref{eq:Ydef} and \eqref{eq:X1def1}, $n\in\Z$. 
Let $\Abar^i_n=\log\Xbar^i_n$.

To define $\Phi$ we need a probability measure $\Gamma$ on $\R$ such that $\int \abs{s}\,\Gamma(ds)<\infty$. Let $M_0=\int s\,\Gamma(ds)$. Given such a $\Gamma$, the mapping $\Phi=\Phi_\Gamma$ 
sends $\mu\in\cM_1(\Omega_A)$ to the distribution $\Phi(\mu)\in\cM_1([-\infty,\infty]^\Z)$ of $(\Abar^0_n)_{n\in\Z}$ induced by $P=\mu\otimes\Gamma^{\otimes\Z}\in\cM_1(\Omega_A\times\Omega_W)$.  

We are interested in shift-invariant fixed points of $\Phi$. To have $\Phi(\mu)\in\cM_1(\Omega_A)$ we need to have $\abs{\Abar^0_n}<\infty$,
$P$-almost surely. This is guaranteed if $\mu$ is satisfies 
 $E^\mu[\abs{A_0^0}]<\infty$
and 
$\mu$-almost surely $E^\mu[A^0_0\,|\,\sI^0]>E[W_0]=M_0$. 
Indeed, as observed above, $\sI\subset\sI^0$ and hence the stochastic process
$(X^0,V)=\{(X^0_n,V_n):n\in\Z\}$ falls in the setting at the beginning of the section. 
Then by \eqref{eq:X1def1} we have $\abs{\Abar^0_n}<\infty$ almost surely.

Given two stationary probability measures $\mu^1,\mu^2$ on $\Omega_A$, let $\sM(\mu^1,\mu^2)$ denote all stationary probability measures on $\Omega_A\times\Omega_A$, with marginals $\mu^1$ and $\mu^2$. 
Recall the definition of the $\bar{\rho}$ distance:
\begin{align}\label{def:rho}
\bar{\rho}(\mu^1,\mu^2) &= \inf_{\lambda\in\sM(\mu^1,\mu^2)} E^\lambda\bigl[|A_0^1 - A_0^2|\bigr].
\end{align}
When $\mu^1$ and $\mu^2$ are ergodic, the infimum may be taken over ergodic measures $\lambda$, by the ergodic decomposition theorem. 
It is shown in \cite[Theorem 8.3.1]{Gra-09} that the $\bar\rho$ distance is a metric and the infimum is 
achieved. 
The following is a positive-temperature analogue of an argument originally due to Chang \cite{Cha-94}. Note that the technical assumption $P(V_0 >c) >0$ for all $c$, required in \cite{Cha-94}, 
is  not needed in positive temperature.

\begin{proposition} \label{prop:contraction}
Let $\mu^1$ and $\mu^2$ be two ergodic probability measures on $\Omega_A$. Assume  $E^{\mu^i}[\abs{A_0^0}]<\infty$ and
$E^{\mu^{i}}[A_0^0]> M_0$, $i\in\{1,2\}$. Then $\bar{\rho}(\Phi(\mu^1),\Phi(\mu^2)) \leq \bar{\rho}(\mu^1,\mu^2)$.
If in addition $\mu^1 \neq \mu^2$ but $E^{\mu_1}[A_0^0]= E^{\mu_2}[A_0^0]$, then
$\bar{\rho}(\Phi(\mu^1),\Phi(\mu^2)) < \bar{\rho}(\mu^1,\mu^1)$.
\end{proposition}

\begin{proof}
Fix an ergodic $\lambda \in \sM(\mu^1,\mu^2)$ and let $P=\lambda\otimes\Gamma^{\otimes\Z}\in\cM_1(\Omega_A\times\Omega_A\times\Omega_W)$. Being a product of an ergodic measure and a product measure, $P$ is also ergodic.
Let $A_n^3 = A_n^1 \vee A_n^2$. Then 
	\[E[A_0^3]=E^\lambda[ A_0^3 ] \ge E^\lambda [A_0^1]=E^{\mu^1}[A_0^0]> M_0.\] 
Thus, the setting at the beginning of the section applies to 
$X^3=\{X_{n}^3= e^{A_{n}^3}:n\in\Z\}$ and $V$. Construct $\Abar^3=\{\Abar_n^3:n\in\Z\}$ as was done above for $\Abar^i$, $i\in\{0,1,2\}$.
Lemma \ref{lem:mono} implies that $P$-almost surely
$\Abar_0^3 \geq \Abar_0^1 \vee \Abar_0^2$. Hence $P$-almost surely and for all $n\in\Z$
\begin{align}\label{coupling}
\begin{split}
&| A_0^1 -  A_0^2| = 2 A_0^1 \vee  A_0^2 -  A_0^1 -  A_0^2 = 2  A_0^3 - A_0^1 -  A_0^2\quad\text{and}\\
&| \Abar_0^1 - \Abar_0^2| = 2 \Abar_0^1 \vee  \Abar_0^2 -  \Abar_0^1 -  \Abar_0^2 \leq 2  \Abar_0^3 - \Abar_0^1 -  \Abar_0^2.
\end{split}
\end{align}
 
By \eqref{eq:means} 
we have $E[\Abar_0^i]=E[\log\Xbar_0^i] = E[\log X_0^i]=E[A_0^i]=E^\lambda[A_0^i]$ for $i\in\{1,2,3\}$.
This and \eqref{coupling}  give
\begin{align*}
E[| \Abar_0^1 - \Abar_0^2|] \le E[2  \Abar_0^3 - \Abar_0^1 -  \Abar_0^2] 
=E^\lambda[2  A_0^3 - A_0^1 -  A_0^2] = E^\lambda[|A_0^1 - A_0^2|].
\end{align*}
The left hand side is greater than $\bar{\rho}(\Phi(\mu^1), \Phi(\mu^2))$. The first claim now follows by taking the infimum over $\lambda$ on the right hand side.

Turning to the second claim, suppose that $E^{\mu_1}[A_{0}^0] = E^{\mu_2}[A_{0}^0]$ and  that $\mu^1 \neq \mu^2$. Let $\lambda\in\cM(\mu^1,\mu^2)$ be an ergodic minimizer of \eqref{def:rho}. Then 
there exists an integer $n > 1$ with
\begin{align*}
\lambda\bigl\{A_{1}^1 < A_{1}^2, A_{m}^1 \leq A_{m}^2, 1 \leq m < n, A_{n}^1 > A_{n}^2\bigr\} > 0.
\end{align*}
Under the event in the above probability, we have
\begin{align*}
\Bigl(\sum_{k=m}^n A_{k}^1\Bigr) \vee \Bigl(\sum_{k=m}^n A_{k}^2\Bigr) < \sum_{k=m}^n A_{k}^3
\end{align*}
whenever $1 \leq m < n$. This is equivalent to
\begin{align} \label{eq:Xstrict0}
\Bigl(\prod_{k=m}^n X_{k}^1 \Bigr)\vee \Bigl(\prod_{k=m}^n X_{k}^2\Bigr) < \prod_{k=m}^n X_{k}^3.
\end{align}
We also have $X_k^1 \vee X_k^2 \leq X_k^3$ for all $k\in\Z$.
Then the representation \eqref{eq:Ydef} of $Y^i$, $i\in\{1,2,3\}$, 
gives $Y_{n+1}^3 < Y_{n+1}^1 \wedge Y_{n+1}^2$ and  
from \eqref{eq:X1def1}, it follows that for $i \in \{1,2\}$
\begin{align*}
\Xbar_{n+1}^i = V_{n+2}\Bigl(1 + \frac{X_{n+1}^i}{V_{n+1}Y_{n+1}^i}\Bigr) <  V_{n+2}\Bigl(1 + \frac{X_{n+1}^3}{V_{n+1}Y_{n+1}^3}\Bigr) = \Xbar_{n+1}^3.
\end{align*}
Thus, $\lambda\bigl(\Abar_{n+1}^1 \vee \Abar_{n+1}^2 < \Abar_{n+1}^3 \bigr) > 0.$ In particular,
\begin{align*}
\bar{\rho}(\Phi(\mu^1),\Phi(\mu^2)) &\leq E^\lambda[|\Abar_0^1 - \Abar_0^2|] = E^\lambda[| \Abar_{n+1}^1 - \Abar_{n+1}^2|] \\
&= E^\lambda[ 2\Abar_{n+1}^1 \vee  \Abar_{n+1}^2 -  \Abar_{n+1}^1 - \Abar_{n+1}^2] \\
&< E^\lambda[2 \Abar_{n+1}^3 - \Abar_{n+1}^1 -  \Abar_{n+1}^2] = E^\lambda[|A_0^1 - A_0^2|] = \bar{\rho}(\mu^1,\mu^2). \qedhere
\end{align*}
\end{proof}

Let $\sM_e^{\alpha}(\Omega_A)$ be the set of ergodic probability measures $\mu\in\cM_1(\Omega_A)$ with marginal mean $E^\mu[A^0_0]=\alpha$. The following is an immediate consequence of the previous proposition.

\begin{corollary}\label{cor:unique} For each $\alpha > M_0$ there exists at most one $\mu \in \sM_e^\alpha(\Omega_A)$ with $\Phi(\mu) = \mu$.
\end{corollary}

\begin{lemma}\label{lem:ergofixed}
Suppose $\mu\in\sM_1(\Omega_A)$ is stationary, $\Phi(\mu) = \mu$, and for some constant $ \alpha > M_0$ we have
\begin{align}\label{erg-lim}
\lim_{n \to \infty} \frac{1}{n} \sum_{m=0}^{n-1} A_{m}^0 = \alpha,\quad \mu\text{-almost surely.}
\end{align}
Then $\mu \in \sM_e^{\alpha}(\Omega_A)$.
\end{lemma}

\begin{proof}
By the ergodic decomposition theorem, there exists $Q_\mu \in \cM_1\bigl(\sM_e^{\alpha}(\Omega_A)\bigr)$ with
\begin{align*}
\mu = \int_{\sM_e^{\alpha}(\Omega_A)}\nu\, Q_\mu(d\nu).
\end{align*}
Let $\phi$ be a bounded measurable function on $\Omega_A$. Then
\begin{align*}
E^{\Phi(\mu)}[\phi(A^0)] = E^{\mu}E^{\Gamma^{\otimes\Z}}[\phi(\Abar^0)] 
&= \int_{\sM_e^\alpha(\Omega_A)} E^{\nu}E^{\Gamma^{\otimes\Z}}[\phi(\Abar^0)]\, Q_\mu(d\nu)\\
&= \int_{\sM_e^\alpha(\Omega_A)} E^{\Phi(\nu)}[\phi(A^0)] \,Q_\mu(d\nu).
\end{align*} 
This is equivalent to $\Phi(\mu) = \int \Phi (\nu)\, Q_\mu(d\nu)$. Since $\Phi(\mu) = \mu$, uniqueness in the ergodic decomposition theorem implies that $Q_\mu \circ \Phi^{-1} = Q_\mu$. 
In particular, $\Phi(\nu)\in\sM_e^\alpha(\Omega_A)$ for $Q_\mu$-almost every $\nu$.
Also, for any $k \in \N$
\begin{align*}
\int_{\sM_e^{\alpha}(\Omega_A)} \bar{\rho}(\nu,\Phi(\nu))\,Q_\mu(d\nu) = \int_{\sM_e^{\alpha}(\Omega_A)} \bar{\rho}(\Phi^k( \nu),\Phi^{k+1}(\nu))\,Q_\mu(d\nu).
\end{align*}
The inequality 
in Proposition \ref{prop:contraction} then implies that
\begin{align*}
Q_\mu\left(\left\{\nu \in \sM_e^{\alpha}(\Omega_A) : \bar{\rho}(\nu,\Phi(\nu)) = \bar{\rho}(\Phi^k(\nu), \Phi^{k+1}(\nu)) \text{ } \forall k \in \N \right\} \right) = 1.
\end{align*}
By the second part of Proposition \ref{prop:contraction} it must be the case that $\nu = \Phi(\nu)$ for $Q_\mu$-almost every $\nu$. Corollary \ref{cor:unique} then implies that $Q_\mu$ is a Dirac mass and so $\mu \in \sM_e^{\alpha}(\Omega_A)$.
\end{proof}

We close this section with a proof of the stationarity mentioned at the end of Section \ref{sec:boundary}.

\begin{lemma}\label{lm:lineBurke}
Fix $\rho>\theta>0$. Assume $\{V_n:n\in\Z\}$ are i.i.d.\ such that $1/V_0$ is gamma-distributed with scale parameter $1$ and shape parameter $\rho$. Assume $X_n$ are i.i.d.\ such that $1/X_0$ is gamma-distributed with scale parameter $1$ and shape parameter $\theta$. Assume the two families of random variables are independent. 
Then $\{\Xbar_n:n\in\Z\}$, defined by \eqref{eq:X1def1}, has the same distribution as 
$\{X_n:n\in\Z\}$.
\end{lemma}

\begin{proof}
Let $Y'_0$ be independent of $\{X_n,V_n:n\in\Z_+\}$ with $1/Y_0'$ being gamma-distributed with scale parameter $1$ and shape parameter $\rho-\theta$. 
Define $\{Y_n':n\in\N\}$ and $\{\Xbar'_n:n\in\Z_+\}$ inductively using
	\[Y'_{n+1}=1+V_nX_n^{-1}Y_n\quad\text{and}\quad\Xbar_n'=V_{n+1}(1+V_n^{-1}X_n/Y_n').\]
The Burke property \cite[Theorem 3.3]{Sep-12-corr} tells us that 
$\{(Y'_{m+n},\Xbar'_{m+n},X_{m+n},V_{m+n}):n\in\Z_+\}$ has the same distribution for all $m\in\Z_+$ and that $\{\Xbar'_n:n\in\Z_+\}$ has the same distribution as $\{X_n:n\in\Z_+\}$.

Using Kolmogorov's extension theorem we can extend the above random variables 
to a family $\{(Y'_n,\Xbar'_n,X_n,V_n):n\in\Z\}$. In particular, the distribution of $\Xbar'$ is the same as that of the process $X$ and $Y'$ and $\Xbar'$ satisfy the above induction for all $n\in\Z$.

By the uniqueness in Lemma \ref{lem:Yunique},  $Y'$ must equal the process $Y$ defined by \eqref{eq:Ydef}, which then implies that $\Xbar'$ is the same as $\Xbar$ defined by \eqref{eq:X1def1}. The claim now follows because we already established that $\Xbar'$ has the same distribution as $X$.
\end{proof}

\section{Proof of Theorems \ref{thm:uniqueness} and \ref{thm:ergodicity}}\label{sec:proofs}

Fix an i.i.d.\ probability measure $\P_0$ on $(\Omega_0,\cF_0)$ with $L^1$ weights.
Let $\P$ be a stationary  future-independent $L^1$ corrector distribution with $\Omega_0$-marginal $\P_0$.  Let $X_n^0=e^{B(ne_1,(n+1)e_1)}$, $Y^0_n=e^{B(ne_1,ne_1+e_2)-\w_{ne_1+e_2}}$, and $V_n=e^{\w_{ne_1+e_2}}$, $n\in\Z$. 
Future independence implies the two processes $X^0$ and $V$ are independent of 
each other. Let $\sI$ be the invariant $\sigma$-algebra for the process $(X^0,V)$.
Let $\mu$ be the distribution of $A^0=\{\log X^0_n:n\in\Z\}$. 
Let $M_0=\E[\w_0]=\E[\log V_0]$. 
Recovery \eqref{recovery} implies that $B(0,e_1)>\w_{e_1}$ and hence $\E[\log A_0^0\,|\,\sI]=\E[B(0,e_1)\,|\,\sI]>\E[\w_0]$.
In particular, $\alpha=\bfm_\P\cdot e_1=\E[B(0,e_1)]>M_0$.

\begin{lemma}\label{lm:aux}
There exists a Borel-measurable map $F:\R^\Z\times\R^{\Z\times\N}\to\R^{\Z^2\times\Z^2}$
such that $\P$-almost surely, 
	\begin{align*}
	&\{B(x,y,\w):x,y\in\Z\times\Z_+\}\\
	&\qquad=F\bigl(\{B(ne_1,(n+1)e_1):n\in\Z\},\{\w_{ne_1+ke_2}:n\in\Z,k\in\N\}\bigr).
	\end{align*}
\end{lemma}

\begin{proof}
Lemma \ref{lm:Lindley} implies that $Y^0$ satisfies \eqref{eq:Y0recursion}. Since it is a stationary almost surely finite process, Lemma \ref{lem:Yunique} implies that $Y^0$ has the representation \eqref{eq:Ydef}.
Then the second equation in \eqref{Lindley} says that $e^{B(ne_1+e_2,(n+1)e_1+e_2)}$ is equal to $\Xbar_n^0$, defined by \eqref{eq:X1def1}.  This argument shows that the process $\{B(ne_1,ne_1+e_2),B(ne_1+e_2,(n+1)e_1+e_2):n\in\Z\}$ is a measurable function of $\{B(ne_1,(n+1)e_1):n\in\Z\}$ and $\{\w_{ne_1+e_2}:n\in\Z\}$.  

Since $\P$ is stationary, we have that $\Abar^0=\{\log\Xbar^0_n:n\in\Z\}$ has the same distribution $\mu$ as $A^0$.  In other words, $\Phi(\mu)=\mu$. This lets us repeat the above procedure inductively to get that $\{B(x,x+e_i):x\in\Z\times\Z_+,i=1,2\}$   is a measurable function of  $\{B(ne_1,(n+1)e_1):n\in\Z\}$ and $\{\w_{ne_1+ke_2}:n\in\Z,k\in\N\}$. We also have $B(x+e_i,x)=-B(x,x+e_i)$, $\P$-almost surely. Then the cocycle property \eqref{cocycle} implies that for $x,y\in\Z\times\Z_+$, $B(x,y)$ is the sum of $B(x_k,x_{k+1})$ along any path with steps $\{\pm e_1,\pm e_2\}$ from $x$ to $y$. The claim of the lemma follows.
\end{proof}

\begin{corollary}\label{cor:unique P}
If $\P'$ is a stationary  future-independent $L^1$ corrector distribution with $\Omega_0$-marginal $\P_0$ and the distributions of $\{B(ne_1,(n+1)e_1):n\in\Z\}$ under $\P'$ and $\P$ match, then $\P'=\P$.
\end{corollary}

\begin{proof}
Lemma \ref{lm:aux} implies that the distributions of $\{B(x,y):x,y\in\Z\times\Z_+\}$ under $\P$ and $\P'$ match. Then stationarity of the two probability measures implies $\P=\P'$.
\end{proof}

\begin{proof}[Proof of Theorem \ref{thm:uniqueness}]
As was mentioned in the proof of Lemma \ref{lm:aux}, $\mu$ is a fixed point of $\Phi$. 
Corollary \ref{cor:unique} says that there exists at most one ergodic such $\mu$.
Corollary \ref{cor:unique P} thus implies that there exists at most one $T_{e_1}$-ergodic $\P$.
Switching $e_1$ and $e_2$ around in the definitions of $X^0$ and $Y^0$ we get the same result for the $T_{e_2}$ shift.  
\end{proof}

\begin{proof}[Proof of Theorem \ref{thm:ergodicity}]
Theorem 4.4 and Lemma 4.5(c) in \cite{Jan-Ras-18-arxiv} 
imply that  $n^{-1}B(0,ne_1)$ converges almost surely to $\bfm_\P\cdot e_1$.
Since $B(0,ne_1)=\sum_{m=0}^{n-1}A_m^0$ we have that \eqref{erg-lim} holds and Lemma \ref{lem:ergofixed} says that $\mu$ is ergodic. Corollary \ref{cor:unique P} says $\P$ is determined by $\mu\otimes\Gamma^{\otimes(\Z\times\Z_+)}$.  Since this is a product of an ergodic measure and a product measure, it is ergodic. Ergodicity of $\P$ under the $T_{e_1}$ shift follows.   A symmetric argument gives the ergodicity under the $T_{e_2}$ shift. 
\end{proof}

\bibliographystyle{aop-no-url}
\bibliography{firasbib2010}

\end{document}